\documentclass[12pt]{article}
 \usepackage{amsmath}
\usepackage{amsfonts}
\usepackage{tikz}
\usepackage[title]{appendix}
\usepackage{subfig}
\usepackage{array}
\usepackage{blkarray}
\usepackage{tabu}
\usepackage{nth}
\usepackage[
  separate-uncertainty = true,
  multi-part-units = repeat
]{siunitx}

\usetikzlibrary{arrows}
\usepackage{epsfig,}
\usepackage{epsfig}
\tikzset{
    vertex/.style = {
        circle,
        draw,
        outer sep = 3pt,
        inner sep = 3pt,
    },edge/.style = {->,> = latex'}
}

\usepackage{amssymb}
\usepackage{enumitem}
\usepackage{amsthm}
\usepackage{dsfont}
\def\diag{\mathop{\rm diag}}
\def\adj{\mathop{\rm adj}}

\def\inertia{\mathop{\rm In}}
\def\rank{\mathop{\rm rank}}
\newcommand{\rr}{\mathbb{R}}
\newcommand{\1}{\mathbf{1}}

\newcommand{\E}{\mathcal{E}}
\newcommand{\n}{\frac{n}{2}}

\def\det{{\rm det}}

\def\cir{{\rm Circ}}

\def\L{\widetilde{L}}
\def\D{\widetilde{D}}
\newtheorem{thm}{Theorem}

\newtheorem{lemma}{Lemma}
\newtheorem{definition}{Definition}

\begin{document}
\begin{center}
\begin{large}
An inverse formula for the distance matrix of a wheel graph with even number of vertices 
\end{large}
\end{center}
\begin{center}
R. Balaji, R.B. Bapat and Shivani Goel \\
\today
\end{center}

\begin{abstract}
Let $n \geq 4$ be an even integer and $W_n$ be the wheel graph with $n$ vertices. The distance $d_{ij}$ between any two distinct vertices $i$ and $j$ of $W_n$ is the length of the shortest path connecting $i$ and $j$. Let $D$ be the $n \times n$ symmetric matrix with diagonal entries equal to
zero and off-diagonal entries equal to $d_{ij}$. In this paper, we find a positive semidefinite matrix $\L$ such that $\rank(\L)=n-1$, all row sums of $\L$ equal to zero and a rank one matrix $ww^T$ such that   
\[D^{-1}=-\frac{1}{2}\L + \frac{4}{n-1}ww^T. \]
An interlacing property between the eigenvalues of $D$ and $\widetilde{L}$ is also proved.
\end{abstract}
{\bf AMS CLASSIFICATION.} 05C50
\section{Introduction}
Consider a simple connected graph $G=(V,\E)$ on $n$ vertices with vertex set $V:=\{1,\dotsc,n\}$ and edge set $\E$. Let $(i,j)$ denote an element in $\E$. If $\delta_i$ is the degree of vertex $i$, then define 
\begin{equation} \label{lap}
    l_{ij}:= \begin{cases}
  \delta_i & {i=j}\\  
    -1&(i,j)\in E\\
    0&\mbox{otherwise}.
    \end{cases}
\end{equation}
The matrix $L:=[l_{ij}]$ is known as the Laplacian of $G$ and is well-studied in the literature. Define $r_{ij}:=l_{ii}^{\dag}+l_{jj}^{\dagger}-2l_{ij}^{\dagger}$, where 
$l_{ij}^{\dag}$ is the $(i,j)$-th entry of the Moore-Penrose inverse of $L$.
If the graph $G$ is represented as an electrical circuit, then $r_{ab}$ is the effective resistance
between the two distinct nodes $a$ and $b$. By the properties of the Laplacian,
$r_{ij}: V \times V \to \rr$ is a metric.  For these reasons, the number $r_{ij}$ is called the resistance distance between the vertices $i$ and $j$ and the matrix
$R:=[r_{ij}]$ is called the resistance matrix of $G$. Among many results on resistance matrices, the formula to compute the inverse of $R$ is more significant. For $i=1,\dotsc,n$, let 
\[\tau_i:=2-\sum_{(i,j) \in E}r_{ij} \]  and $\tau$ be the $n \times 1$ vector with components $\tau_1,\dotsc,\tau_n$. Then,
\begin{equation} \label{resis}
R^{-1}=-\frac{1}{2}L+ \frac{1}{\tau'R\tau} \tau \tau'.
\end{equation} 
See Theorem 9.2 in \cite{bapat}.
 Resistance matrices are mathematically tractable. For example, it is easy to compute its inertia and
determinant. The definition of $r_{ij}$ ensures that the resistance matrices are Euclidean distance matrices. These matrices are known to have many interesting properties. See for e.g., \cite{alfakih}. 

In this paper, we consider the natural distance $d_{ij}$ which is the length of the shortest path connecting any two vertices $i$ and $j$ in $G$. We say that $D:= [d_{ij}]$ is the distance matrix of $G$. It is easy to see that $d_{ij}: V \times V \to \rr$ is a metric. In contrast to resistance matrices, distance matrices do not have many nice properties. 
 For example, distance matrices of connected graphs like cycles with even number of vertices are not invertible and are not Euclidean distance matrices. 
Unlike the resistance, there is no general identity that connects the natural distance and the Laplacian matrix of a graph.
 Nevertheless, distance matrices have wide applications. For example in chemistry, the classical distance $d_{ij}$ represents the structure of a molecule: see \cite{Klein}. Eigenvalues of distance matrices are applied to solve data communication problems in \cite{Graham}. 
Applications of distance matrices to biology are seen in \cite{jak}. 
 In view of these reasons, it is significant to study distance matrices.

If $D$ is the distance matrix of  a tree on $n$-vertices, then according to a classical result of Graham and Lov\'{a}sz,  
\begin{equation} \label{gl}
D^{-1}=-\frac{1}{2}L + \frac{1}{2(n-1)} \tau \tau', 
\end{equation}
where 
$L$ is the Laplacian matrix and $\tau:=(2-\delta_1,\dotsc,2-\delta_n)'$.
In this spirit, there are inverse formula for distance matrices of some connected graphs,
see for e.g., \cite{bapat_kirk, sivasu, hou}. 
We investigate distance matrices of wheel graphs in this paper. If $D$ is the distance matrix of a wheel graph on $n$ vertices, then
\[ \det(D) = \begin{cases}
1-n &~\mbox{if}~n~\mbox{is even}\\
0 &~\mbox{if}~n~\mbox{is odd}.
\end{cases}\]
See Theorem 7 in \cite{zhang}. By performing certain numerical experiments, we observed that for a wheel graph with even number of vertices, the inverse of the distance matrix has a simple formula as given in $(\ref{resis})$ and $(\ref{gl})$. The following is the main result of this paper. If $D$ is the distance matrix of a wheel graph with even number of vertices, then we find a positive semidefinite matrix $\L$ with  rank $n-1$ and all row sums equal to zero such that 
\[D^{-1} = -\frac{1}{2}\L+\frac{4}{n-1}ww',\]
where $w=\frac{1}{4}(5-n,1,\dotsc,1)'.$
In addition, we prove an interlacing property between the eigenvalues of $D$ and $\L$.
\section{Preliminaries}
We fix notation and mention a few results that are needed in the sequel.
\begin{enumerate}
\item[{\rm (P1)}] All vectors are assumed to be column vectors. If $A$ is a matrix, we use $A'$
to denote its transpose. The circulant matrix $C$ specified by a vector 
$c=(c_1,\dotsc,c_n)'$ in $\rr^{n}$ is the Toeplitz matrix 
\begin{equation*}
C := \left[
{\begin{array}{rrrrrr}
c_1 & c_2 & c_3 & \ldots & c_n \\
c_n & c_1 & c_2 & \ldots & c_{n-1} \\
c_{n-1} & c_n & c_1 &  \ldots &c_{n-2} \\
\vdots & \vdots & \vdots &  \ldots &\vdots \\
c_2 & c_3 & c_4 &  \ldots &c_1
\end{array}}
\right].
\end{equation*}
We now write $C=\cir(c')$.
Let $T:\rr^n \to \rr^n$ be the shift operator $T(v_1,\dotsc,v_n)'=(v_n,v_1,\dotsc,v_{n-1})'$. Then the $k$-th row of $\cir(c')$ 
is equal to $(T^{k-1}c)'$.

\item[({\rm P2)}] Let $n \geq 4$. The wheel graph with $n$-vertices is denoted by $W_n$. Let $C_{n-1}$ be the subgraph
of $W_n$ which is a cycle of length $n-1$. Without loss of generality, we assume that the vertices of $W_n$ are labelled as follows.
The vertices of $C_{n-1}$ are labelled $2,3,\dotsc,n$ anticlockwise and the hub of $W_n$ is labelled $1$. See for e.g., Figure \ref{fig_W6}.

\item[({\rm P3)}] We use the notation $\1$ to denote the vector of all ones and $I$ to denote the identity matrix. The number of components in $\1$ and the order of $I$
will be clear from the context.
\item[({\rm P4})] Let $D$ be the distance matrix of $W_n$. By labelling the vertices of $W_n$ as mentioned in (P2), $D$ has the form
\begin{equation*}\label{E_block_form_D}
    D = \left[ \begin{array}{cc}
        0 & \1' \\
        \1 & \widetilde{D}
    \end{array}\right],
\end{equation*}
where $\widetilde{D}=\cir(0,1,\underbrace{2,\dotsc,2,}_{(n-4)}1)$. We note that $\widetilde{D} \1=2(n-3) \1$.

\item[{\rm (P5)}] Let $C_1$ and $C_2$ be circulant matrices of the same order. If $C_1=\cir(c')$, then $C_1C_2=\cir(c'C_2)$. 

\item[{\rm (P6)}]  The following identity will be useful in the sequel. If $n$ is even, then
$$\sum_{k=1}^{\frac{n}{2}-1} (-1)^k {(n-1-2k)}=\dfrac{2-n}{2}.$$

\item[{\rm (P7)}] 
Let $x=(x_1,\dotsc,x_{n-1})' \in \rr^{n-1}$, where $n \geq 4$ is even. We say that $x$ follows symmetry in its last $n-2$ coordinates if 
\[x_i = x_{n+1-i}~~\mbox{for all}~~ i=2,3,\dotsc,n-1.\]
In other words, $x \in \rr^{n-1}$ follows symmetry in its last $n-2$ coordinates if and only if
\[x = (x_1,x_2,x_3,\dotsc,x_{\frac{n}{2}},x_{\frac{n}{2}},\dotsc,x_3,x_2)'.\]
\item[{\rm (P8)}]  Let $A$ be an $n \times n$ matrix. If $u$ and $v$ belong to $\rr^n$, then $\det(A+uv') = \det(A)+v'\adj(A)u$. This result is well known as matrix determinant lemma.
\item[{\rm (P9)}] An $n \times n$ non-negative symmetric matrix $A$ is called a Euclidean distance matrix, if all the diagonal entries are equal to zero and $x'Ax \leq 0$ for all $x \in \{\1\}^{\perp}.$
\end{enumerate}
\subsection{Special Laplacian for $W_n$}
We introduce a special Laplacian matrix for wheel graphs with even number of vertices.
This definition is motivated from numerical computations. 
\begin{definition}\label{D_laplacian} \rm
Let $n \geq 4$ be even. For each $k \in \{1,2,\dotsc,\frac{n}{2}-1\}$, define a vector ${c^k} := (c_{1}^{k},\dotsc,c_{n-1}^{k}) '$ in $\rr^{n-1}$ by
\begin{equation*}
    c^k_j := \begin{cases}
    1 & j=k+1~\mbox{or}~j=n-k \\ 
    0 &\mbox{otherwise}.
    \end{cases}
\end{equation*}
We now define
\begin{equation*}
    \L := \dfrac{(n-1)}{2}I - \dfrac{1}{2} \left[{\begin{array}{cc}
        0 & \1'  \\ 
        \1 & 0
    \end{array}}\right] + \sum_{k=1}^{\frac{n}{2}-1} (-1)^k \dfrac{(n-1)-2k}{2} \left[{\begin{array}{cc}
        0 & 0  \\
        0 & C_k
    \end{array}}\right],
\end{equation*}
where
$C_k:=\cir({c^k}^{'})$. We say that $\L$ is the special Laplacian of $W_n$.
\end{definition}

\section{Results}

\subsection{Inverse formula}
The following theorem is our main result.
\begin{thm} \label{inverse}
Let $n \geq 4$ be an even integer and $D$ be the distance matrix of $W_n$. Define $w \in \rr^n$ by
$w:=\frac{1}{4}(5-n,1,\dotsc,1)'$. Then,
\[D^{-1} = -\frac{1}{2} \L+\dfrac{4}{n-1}ww',\] where $\L$ is the special Laplacian of $W_n$.
\end{thm}
We now illustrate Theorem \ref{inverse} for $W_6$.
\subsection{Illustration for $W_6$} \label{W6_example}
We label the six vertices of $W_6$ as mentioned in (P2): see Figure $\ref{fig_W6}$.
\begin{figure}
\centering
\begin{tikzpicture}[shorten >=1pt, auto, node distance=3cm, ultra thick,
   node_style/.style={circle,draw=black,fill=white !20!,font=\sffamily\Large\bfseries},
   edge_style/.style={draw=black, ultra thick}]
\node[vertex] (1) at  (0,0) {$1$};
\node[vertex] (2) at  (2,-2) {$2$};
\node[vertex] (3) at  (2,0) {$3$}; 
\node[vertex] (4) at  (0,2) {$4$};  
\node[vertex] (5) at  (-2,0) {$5$};  
\node[vertex] (6) at  (-2,-2) {$6$};  
\draw  (1) to (2);
\draw  (1) to (3);
\draw  (1) to (4);
\draw  (1) to (5);
\draw  (1) to (6);
\draw  (2) to (3);
\draw  (3) to (4);
\draw  (4) to (5);
\draw  (5) to (6);
\draw  (6) to (2);
\end{tikzpicture}
\caption{Wheel graph $W_6$} \label{fig_W6}
\end{figure}
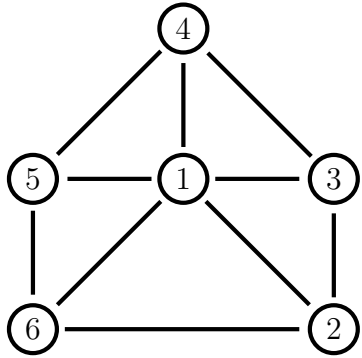
The special Laplacian $\L$ is now given by
\begin{equation*}
      \L=\frac{1}{2}\left[
{\begin{array}{rrrrrr}
5& -1&  -1 & -1& -1 & -1 \\
-1 & 5 & -3 & 1& 1 & -3 \\
-1 & -3 & 5 & -3 & 1 & 1 \\
-1 & 1 & -3 & 5 & -3 & 1 \\
-1 & 1 & 1 & -3 & 5 & -3 \\
-1 & -3& 1 & 1 & -3 & 5
\end{array}}
\right].
\end{equation*}
The distance matrix $D$ of $W_6$ is given by
\begin{equation*}
      D=\left[
{\begin{array}{rrrrrr}
0& 1&  1 & 1& 1 & 1 \\
1 & 0 & 1 & 2& 2 & 1 \\
1 & 1 & 0 & 1 & 2 & 2 \\
1 & 2 & 1 & 0 & 1 & 2 \\
1 & 2 & 2 & 1 & 0 & 1 \\
1 & 1& 2& 2 & 1 & 0
\end{array}}
\right].
\end{equation*}
Let $w=\frac{1}{4}(-1,1,\dotsc,1)'$.
By an easy computation,
\[ -\frac{1}{2}\L + \frac{4}{5} ww'=1/5 \left[{\begin{array}{rrrrrr}
-6 & 1 & 1 & 1 & 1 & 1 \\

1 & -6 &4 & -1& -1& 4 \\

1 & 4 & -6 & 4 & -1 & -1 \\

1 & -1 & 4 & -6 & 4 & -1 \\

1 & -1 & -1 & 4 & -6 & 4 \\

1 & 4 & -1 & -1 & 4 & -6 \\
\end{array}}
\right], \]
which is the inverse of $D$.

\subsection{Proof of Inverse formula}
For $W_4$, the result can be verified directly and for $W_{6}$, the result is true from the above illustration.
In the rest of the paper, we assume $n \geq 8$. We need some elementary lemmas to prove the main result.
Recall from (P4) that the first column of $\D$ is equal to $(0,1,2\dotsc,2,1)'$. We fix the notation $u$ to denote this vector. We shall use the identity in (P6) repeatedly.
\begin{lemma} \label{M}
\[\L D=
\left[\begin{array}{cccc}
\dfrac{1-n}{2} & \dfrac{5-n}{2}\1\\
\\
\dfrac{1}{2}\1' & M
\end{array}\right],\]
where
\[M:=\cir \big(\dfrac{n-1}{2}u'-\frac{1}{2}\1'+\sum_{k=1}^{\frac{n}{2}-1} (-1)^{k} \frac{(n-1)-2k}{2}{c^k}^{'} \D\big).\]
\end{lemma}
\begin{proof}
Multiplying $\L$ and $D$, we get 
\[\L D=
\left[\begin{array}{cccc}
\frac{1-n}{2} & A\\
B & M
\end{array}\right],\]
where 
\[A = \frac{n-1}{2}\1'-\frac{1}{2} \1'\D ,\]
\[B =\dfrac{n-1}{2} \1+\sum_{k=1}^{\frac{n}{2}-1} (-1)^k \dfrac{(n-1)-2k}{2}C_k \1,\]
and
\[M= \frac{n-1}{2} \D-\dfrac{1}{2} \1 \1'+\sum_{k=1}^{\frac{n}{2}-1} (-1)^k \dfrac{(n-1)-2k}{2}C_k\D.\] 
Since $\D=\cir (u' )$, $\1 \1'=\cir(\1')$ and $C_k\D=\cir({c^k}^{'} \D)$, we get
\[M=\cir \big(\dfrac{n-1}{2}u'-\frac{1}{2}\1'+\sum_{k=1}^{\frac{n}{2}-1} (-1)^{k} \frac{(n-1)-2k}{2}{c^k}^{'} \D\big).\] 
As
     $\D \1 = {2}{(n-3)} \1$, we have
\begin{equation*}\label{E_A}
\begin{aligned}
A &= \dfrac{n-1}{2}\1'-\dfrac{1}{2}\1'\widetilde{D} \\ 
 &= \dfrac{n-1}{2} \1'-(n-3) \1'
\\ &=  \dfrac{5-n}{2}\1'.
\end{aligned}
\end{equation*} 
 Each $C_k$ is a circulant matrix specified by a vector in $\rr^{n-1}$ with exactly two ones and remaining entries equal to zero. Therefore, 
 \[C_k \1 = 2 \1~~~\mbox{for all}~ k=1,2,\dotsc,\n-1.\]
 Thus,
\begin{equation*}
\begin{aligned}
B=\dfrac{n-1}{2} \1+\sum_{k=1}^{\frac{n}{2}-1} (-1)^k (n-1-2k) \1.
\end{aligned}
\end{equation*} 
By the identity in (P6), we now get 
\begin{equation*}\label{E_B}
B =  \dfrac{n-1}{2}\1-\dfrac{n-2}{2} \1=  \dfrac{1}{2} \1.
\end{equation*}
This proves the lemma.
\end{proof}

To simplify $M$, we compute the row vectors ${c^k}^{'} \D$ for each $k=1,\dotsc,\n-1$ precisely. This is done in the next three lemmas.
\begin{lemma} \label{first_row}
${c^{1}}'\widetilde{D}=({2,2},3,\underbrace{4,\dotsc,4}_{n-6},3,2)$
and ${c^{1}}'\D$ follows symmetry in its last $n-2$ 
co-ordinates.
\end{lemma}

\begin{proof}
By definition $\ref{D_laplacian}$, ${c^1}'=(0,1,0,\dotsc,0,1)'$.  So, ${c^{1}}'\widetilde{D}$ is the sum of the second row and the last
row of $\widetilde{D}$. Let $x$ be the second row and $y$ be the last row of $\widetilde{D}$. Then,
\[x=(1,0,1,2,\dotsc,2)~~\mbox{and}~~y=(1,2,\dotsc,2,1,0).\]
Now, $x+y=({2,2},3,\underbrace{4,\dotsc,4}_{n-6},3,2)$. 
To verify ${c^{1}}'\D$ follows symmetry in its last $n-2$ co-ordinates is direct. This completes the proof.
\end{proof}

\begin{lemma} \label{n/2_row}
Let $p:=c^{\frac{n}{2}-1}$. Then,
\[p'\widetilde{D}=(\underbrace{4,\dotsc,4}_{\frac{n}{2}-2},3,1,1,3,\underbrace{4,\dotsc,4}_{\frac{n}{2}-3}), \]
and $p'\D$ follows symmetry in its last $n-2$ coordinates.
\end{lemma}
\begin{proof}
By definition \ref{D_laplacian},
\[p'=(\underbrace{0,\dotsc,0}_{\frac{n}{2}-1},1,1,\underbrace{0,\dotsc,0}_{\frac{n}{2}-2}).\]
Hence, $p'\D$ is the sum of $\n$-th and $(\n +1)$-th rows of $\widetilde{D}$. Let these two rows be $x$ and $y$,
respectively.  Because $\widetilde{D}$ is circulant, $x'=T^{\frac{n}{2}-1}u$ and $y'=T^{\frac{n}{2}}u$, where $u=(0,1,2,\dotsc,2,1)'$ and $T$ is the shift operator defined in (P1). Thus,
\[x=(\underbrace{2,\dotsc,2}_{\frac{n}{2}-2},1,0,1,\underbrace{2,\dotsc,2}_{\frac{n}{2}-2}),\]
\[y=(\underbrace{2,\dotsc,2}_{\frac{n}{2}-1},1,0,1,\underbrace{2,\dotsc,2}_{\frac{n}{2}-3}). \]
Therefore, $$p'\D=x+y=(\underbrace{4,\dotsc,4}_{\frac{n}{2}-2},3,1,1,3,\underbrace{4,\dotsc,4}_{\frac{n}{2}-3}).$$ 
Define $s:=p' \D$. Then, \[s_i = \begin{cases}
1& \mbox{for}~i={\n},{\n+1}\\
3&\mbox{for}~i={\n-1},{\n+2}\\
4&\mbox{otherwise}.
\end{cases}\]
The above equation shows that $s$ follows symmetry in 
its last $n-2$ coordinates.
The proof is complete.
\end{proof}
\begin{lemma}\label{1<k}
Let $1<k<\n-1$. Define $q^k:={c^k}'\widetilde{D}$.  If $q^k=(q_{1}^k,\dotsc,q_{n-1}^k)$, then
\[q_j^k=
    \begin{cases}
    2 & \mbox{if~}j=k+1,n-k \\
    3 & \mbox{if~}j=k,k+2,n-k-1,n-k+1\\
    4 & \mbox{otherwise.}
    \end{cases}\]
Furthermore, each $q^k$ follows symmetry in its last $n-2$ coordinates.
\end{lemma}
\begin{proof}
Let $1<k<\n-1$ be fixed. Then, 
\begin{equation*}
    c^k_j = \begin{cases}
    1 & j=k+1~\mbox{or}~j=n-k \\ 
    0 &\mbox{otherwise}.
    \end{cases}
\end{equation*}
Hence, $q^k$ is the sum of $(k+1)$-th and $(n-k)$-th rows of $\widetilde{D}$. Let these rows be $x$ and $y$, respectively.
Now,  $x'=T^{k}u$ and $y'=T^{n-k-1}u$.
This gives 
\[x=(\underbrace{2,\dotsc,2}_{k-1},1,0,1,\underbrace{2,\dotsc,2}_{n-k-3})'. \]
\[y=(\underbrace{2,\dotsc,2}_{n-k-2},1,0,1,\underbrace{2,\dotsc,2}_{k-2})'. \]
We now compute $x+y$.
Since $n$ is even and $1<k<\n-1$, we have 
\[n-k-2 \geq k+2.\]
This immediately gives 
\begin{equation} \label{s1}
(x+y)_{j}=
    \begin{cases}
    4 & \mbox{if~}j=1,\dotsc,k-1 \\
    3 & \mbox{if~}j=k,k+2\\
    2 & \mbox{if~} j=k+1.
    \end{cases}
    \end{equation}
 Let $k+2<j\leq n-k-2$. Then, $x_{j}=y_{j}=2$. Hence in this case 
 \begin{equation}\label{snew}
(x+y)_{j}=4.     
 \end{equation}
We note that  \[y_j=
    \begin{cases}
    1 & \mbox{if~}j=n-k-1,n-k+1 \\
    0 & \mbox{if~}j=n-k.
    \end{cases}\]
Since $x_{j}=2$ for all $j>k+2$ and $n-k-1>k+2 $, we have $x_{j}=2$ for all $j \geq n-k-1$. Now, it follows that 
\begin{equation} \label{s2}
(x+y)_j=
    \begin{cases}
    3 & \mbox{if~}j=n-k-1,n-k+1 \\
    2 & \mbox{if~}j=n-k.
    \end{cases}
    \end{equation}
Finally, for all $j >n-k+1$, $x_{j}=2$ and $y_{j}=2$. Thus,
\begin{equation} \label{s3}
(x+y)_{j}=4~~\mbox{if~} n-k+1<j \leq n-1.
\end{equation}
By $(\ref{s1})$, $(\ref{snew})$, $(\ref{s2})$ and $(\ref{s3})$, we get \[q_j^k=
    \begin{cases}
    2 & \mbox{if~}j=k+1,n-k \\
    3 & \mbox{if~}j=k,k+2,n-k-1,n-k+1\\
    4 & \mbox{otherwise.}
    \end{cases}\]
To show that the vector $q^k$ follows symmetry in its last $n-2$ co-ordinates, we verify the equations:
 \[ q^k_{i}=q^k_{n+1-i}~~ \forall i=2,\dotsc,n-1.\] We consider three possible cases.
\begin{enumerate}
    \item[(i)] $q^k_i=2$. This means $i \in \{k+1,n-k\}$. We note that $i=k+1$ if and only if $n+1-i=n-k$ . Since $q^k_{k+1}=q^k_{n-k}=2$, we have $q^k_i=q^k_{n+1-i}$ in this case.
    \item[(ii)] $q^k_i=3$.  This implies $$i \in \{k,k+2,n-k-1,n-k+1\}.$$ It is easy to see that $i=k$ if and only if $n+1-i=n-k+1$ and 
$i=k+2$ if and only $n+1-i=n-k-1$. Therefore, $q^k_{i}=q^k_{n+1-i}$.
\item[(iii)] $q^k_i=4$. 
This implies $i \notin \{k+1,n-k\}$ and hence, $n+1-i \notin \{k+1,n-k\}$. Therefore
$q^k_{n+1-i} \neq 2$.
By a similar argument,  $q^k_{n+1-i} \neq 3$. Hence, $q^k_{n+1-i}=4$.
\end{enumerate}
The proof is complete.
\end{proof}
To simplify $M$, we compute
$$f:=\sum_{k=1}^{\n-1} (-1)^{k} \frac{(n-1-2k)}{2}{c^k}'\D.$$
For $1\leq k \leq \n-1$, define $q^k:={c^k}'\D$. We shall write $q^k=(q_{1}^k,\dotsc,q_{n-1}^k)$ and $f:=(f_{1},\dotsc,f_{n-1})$. Now $f$ is the row vector
\[f=\sum_{k=1}^{\n-1} (-1)^{k} \frac{(n-1-2k)}{2}(q_{1}^k,\dotsc,q_{n-1}^{k}).\]
We now compute $f$ precisely.

\begin{lemma} \label{f1}
$f_{1}=-1.$
\end{lemma}

\begin{proof}
By Lemma \ref{first_row}, \ref{n/2_row} and $\ref{1<k}$, we have 
\[q_{1}^{1}=2,~~q_{1}^{k}=4~~\forall k=2,\dotsc,\n-1. \]
In view of this,
\begin{equation} \label{x}
    \begin{aligned}
        f_{1}&=\sum_{k=1}^{\n-1} (-1)^k \dfrac{(n-1-2k)}{2}q_{1}^{k}   \\
        &= (3-n) + 2\sum_{k=2}^{\frac{n}{2}-1} (-1)^k (n-1-2k).
\end{aligned}
\end{equation}        
By the identity in (P6),  $$2\sum_{k=1}^{\frac{n}{2}-1} (-1)^k (n-1-2k)=2-n.$$
So, 
\[2\sum_{k=2}^{\frac{n}{2}-1} (-1)^k (n-1-2k) =(2-n)-2(3-n).\]
Substituting the above in $(\ref{x})$, we get
\[f_1=(3-n)+(2-n)-2(3-n)=-1.\] The proof is complete.
\end{proof}

\begin{lemma} \label{f2}
$f_{2}=\frac{3-n}{2}$.
\end{lemma}
\begin{proof}
We need to show that
\[\sum_{k=1}^{\n-1} (-1)^k \dfrac{(n-1-2k)}{2} q_{2}^{k} = \frac{3-n}{2}.\]
By Lemma \ref{first_row}, \ref{n/2_row} and \ref{1<k},
\[q_{2}^{k} = \begin{cases}
2 & \mbox{if}~k=1\\
3 & \mbox{if}~k=2\\
4 &\mbox{otherwise.}
\end{cases}\]
This gives,      
      \begin{equation*}
    \begin{aligned}
        f_2 &= \bigg((-1) \dfrac{(n-1-2)}{2} \times 2 \bigg)+ \bigg((-1)^2 \dfrac{(n-1-4)}{2} \times 3\bigg)
       +\bigg(\sum_{k=3}^{\n-1} (-1)^k \dfrac{(n-1-2k)}{2} \times 4\bigg)\\ 
        \\&= (3-n) + \frac{3}{2}(n-5) + 2\sum_{k=1}^{\frac{n}{2}-1} (-1)^k {(n-1-2k)}+4 .\\ 
        \end{aligned}
        \end{equation*}
Using identity (P6),
$$f_{2}=3-n+\frac{3}{2}(n-5) + 2-n +4. $$
Simplifying, we get $f_2=\frac{3-n}{2}$.
\end{proof}

\begin{lemma} \label{f3}
Let $2<j \leq \n-1$. Then, $f_j=2-n$.
\end{lemma}
\begin{proof}
Since $2<j \leq \n-1$, by Lemma \ref{first_row}, $\ref{n/2_row}$ and $\ref{1<k}$, we have
\[q^1_j= \begin{cases}
3&~\mbox{if}~j=3\\
4&~\mbox{otherwise},
\end{cases}\]
\[q^{\frac{n}{2}-1}_j = \begin{cases}
3&~\mbox{if}~j=\frac{n}{2}-1\\
4&~\mbox{otherwise},
\end{cases}\]
and for $1<k<\frac{n}{2}-1$
\[q^k_j= \begin{cases}
2&~\mbox{if}~j=k+1\\
3&~\mbox{if}~j=k,k+2\\
4&~\mbox{otherwise}.
\end{cases}\]
From the above, we see that
\[q^k_j= \begin{cases}
2&~\mbox{if}~j=k+1~\mbox{and}~1<k<\frac{n}{2}-1\\
3&~\mbox{if}~j=3~\mbox{and}~ k=1 \\
3&~\mbox{if}~j= k= \frac{n}{2}-1 \\
3&~\mbox{if}~j=k,k+2~\mbox{and}~1<k<\frac{n}{2}-1  \\
4&~\mbox{otherwise}.
\end{cases}\]
This gives
\begin{equation}\label{qjk}
    q^k_j= \begin{cases}
2&~\mbox{if}~k=j-1\\
3&~\mbox{if}~k=j,j-2 \\
4&~\mbox{otherwise}.
\end{cases}
\end{equation}
We need to compute 
\[f_j= \sum_{k=1}^{\n-1} (-1)^k \dfrac{(n-1)-2k}{2} q_{j}^{k}\]
for $2<j \leq \n-1$. Define \[\alpha:=(-1)^{j-1} \dfrac{(n-1-2(j-1))}{2},\]
 $$\beta:=(-1)^{j} \dfrac{(n-1-2j)}{2},$$ \[\gamma:=(-1)^{j-2} \frac{(n-1-2(j-2))}{2}.\]
It is easy to see that 
\begin{equation} \label{pp}
2 \alpha + \beta + \gamma=0.
\end{equation}
Define $\Omega:=\{j-2,j,j-1\}$. By  (\ref{qjk}), we have
\[f_{j}=2 \alpha + 3 \beta + 3 \gamma +4 \sum_{k \notin \Omega}(-1)^{k}\frac{(n-1-2k)}{2}.\]
Since \[\sum_{k \notin \Omega} (-1)^{k} \frac{(n-1-2k)}{2}=\sum_{k=1}^{\n-1}(-1)^{k} \frac{(n-1-2k)}{2}-\alpha-\beta-\gamma, \]
by the identity in (P6), 
\[f_{j}=2 \alpha + 3 \beta + 3 \gamma + (2-n)-4(\alpha+\beta+\gamma). \]
After simplification, we get
$$f_{j}=-2 \alpha -\beta -\gamma+2-n. $$
From $(\ref{pp})$,  we now have 
$f_{j}=2-n$.
\end{proof}

\begin{lemma} \label{f4}
$f_{\n}=2-n.$
\end{lemma}
\begin{proof}
In view of Lemma \ref{first_row}, \ref{n/2_row} and \ref{1<k}, we have
\begin{equation*}
    q_{\n}^{k}=
    \begin{cases}
    1 &\mbox{if~}k=\n-1\\
    3 & \mbox{if~}k=\n-2.\\
    4 & \mbox{otherwise.} \\
 \end{cases}
\end{equation*}
Now \[f_{\n}=\sum_{k=1}^{\n-1}(-1)^k \frac{(n-1-2k)}{2}q_{\n}^{k}. \] So, 
\[f_{\n}= 4\sum_{k=1}^{\n-3}(-1)^k \frac{(n-1-2k)}{2} + 3(-1)^{\n-2} \frac{(n-1-2(\n-2))}{2}+(-1)^{\n-1} \frac{(n-1-2(\n-1))}{2}.\]
It is easy to see that 
\[\frac{n-1-2(\n-2)}{2}=3/2~~~\mbox{and}~~~\frac{n-1-2(\n-1)}{2}=1/2.\]
Hence 
\[f_{\n}= 4\sum_{k=1}^{\n-3}(-1)^k \frac{(n-1-2k)}{2} + 4 (-1)^{\n} .\]
The identity in (P6) now gives
$$\sum_{k=1}^{\n-3}(-1)^k \frac{(n-1-2k)}{2}=\frac{2-n}{4}-(-1)^{\n}.$$
The last two equations imply that $f_{\n}=2-n$.
\end{proof}

\begin{lemma} \label{f-vector}
$f=(-1,\frac{3-n}{2},2-n,\dotsc,2-n,\frac{3-n}{2})$.
\end{lemma}
\begin{proof}
By Lemma \ref{f1}, $\ref{f2}$, $\ref{f3}$ and $\ref{f4}$, we have 
\[(f_1,\dotsc,f_{\n}) =(-1,\frac{3-n}{2},2-n,\dotsc,2-n).\]
By Lemma  \ref{first_row}, $\ref{n/2_row}$ and \ref{1<k}, 
$q^k$ follows symmetry in its last $n-2$ coordinates for each $k=1,2,\dotsc,\n-1$.
 So, $f$
follows symmetry in its last $n-2$ co-ordinates. Thus,
\[f=(-1,\frac{3-n}{2},2-n,\dotsc,2-n,\frac{3-n}{2}). \]
The proof is now complete.
\end{proof}
We now prove Theorem $\ref{inverse}$.
\begin{proof}
In view of Lemma \ref{M} and \ref{f-vector},
\[\L D=
\left[\begin{array}{cccc}
\frac{1-n}{2} & \frac{5-n}{2} \1' \\
\\
\frac{1}{2}\1 & M
\end{array}\right],\]
where
 \[M=\cir (\dfrac{n-1}{2}u'-\frac{1}{2} \1{'}+f).\]
 Since \(u=(0,1,2,\dotsc,2,1)'\) and 
 \(f=(-1,\frac{3-n}{2},2-n,\dotsc,2-n,\frac{3-n}{2}), \)
\[\dfrac{n-1}{2}u'-\frac{1}{2} \1'+f'=(-\frac{3}{2},\frac{1}{2},\dotsc,\frac{1}{2}).\] 
By an easy manipulation we have,
\[\cir(-\frac{3}{2},\frac{1}{2},\dotsc,\frac{1}{2})=-2 I+ \frac{1}{2} \1 \1'. \]
By setting $w=\frac{1}{4}(5-n,1,\dotsc,1)'$, we deduce that
\begin{equation} \label{ld2}
\L D +2 I=2 w\1'.
\end{equation}
Another direct verification gives 
\begin{equation} \label{ld3}
Dw=\frac{n-1}{4} \1.
\end{equation}
As $\det(D)=1-n$, $D$ is non-singular.
 Hence by $(\ref{ld2})$, we have 
 \[(2I-2w \1')D^{-1} = -\L.\]
So,
\begin{equation} \label{inv}
    2D^{-1} = -\L+2w \1'D^{-1}.
   \end{equation}
    By (\ref{ld3}),
\[\1'D^{-1} = \frac{4}{n-1} w'.\]
Now equation (\ref{inv}) gives
\[D^{-1} = -\frac{1}{2}\L+\frac{4}{n-1} ww'.\]
This completes the proof.

\end{proof}

\subsection{Properties of the special Laplacian matrix}
The usual Laplacian matrix (see equation (\ref{lap})) 
is positive semidefinite, has rank $n-1$, all row sums equal to $0$ and any cofactor is equal to the number
of spanning trees in the graph.
We show that the special Laplacian matrix $\L$ also has these properties and any cofactor of $\L$ is equal to $2^{n-3}$.
\begin{thm}
Column sums of $\L$ are zero and $\rank(\L)=n-1$.  
\end{thm}
\begin{proof}
We recall equation $(\ref{ld2})$,
\begin{equation*} 
\L D +2 I=2 w\1'.
\end{equation*}
Since $\1'w=1$, the above equation gives 
$\1' \L D=0$.
Let $p \in \rr^n$ be a non-zero vector such that $p' \L D=0$.
In view of (\ref{ld2}),  we have
\[p'(-2I+2w\1') = 0.\]
This gives
\[p'=(p'w)\1'.\]
So, $p$ is a multiple of $\1$. Thus,
nullity of $\L D$ is one and hence $\rank(\L D)=n-1$.
As $D$ is non-singular, $\rank(\L)=n-1$.
Since $\1' \L D=0$ if and only if $\1' \L=0$, all the column sums of $\L$ are zero. The proof is complete.
\end{proof}

\begin{thm}
$\L$ is positive semidefinite.
\end{thm}
\begin{proof}
Since $\rank(\L)=n-1$ and $\L \1=0$, $\L \L^{\dag}$ is a symmetric idempotent matrix with null space equal to $\mbox{span}\{\1\}$.
Thus, $\L \L^{\dag}=I-\frac{\1 \1'}{n}$. Define $J:=\1 \1'$ and $P:=I-\frac{J}{n}$. By the identity $$\L D=-2I+2w \1',$$ we get 
$PDP=-2\L^{\dag}$. Let $D=[d_{ij}]$ and $\L^{\dag}:=[a_{ij}]$. It is now easy to get the relation
\[d_{ij}=a_{ii}+a_{jj}-2a_{ij}. \]
From the above equation, 
\begin{equation} \label{ldj}
D=\diag(\L^{\dag})J+ J\diag( \L^{\dag}) -2 \L^{\dag}. 
\end{equation}
By Theorem 12 in \cite{Jakli}, $x' D x \leq 0$ for all $x \in \{\1\}^{\perp}$.
Now, $(\ref{ldj})$ implies that $x' \L^{\dag} x \geq 0$ for all $x \in \{\1\}^{\perp}$. We know that $\rank(\L)=n-1$ and $\L\1=0$. By the properties of Moore-Penrose inverse, we deduce that $x' \L x \geq 0$ for all $x \in \{\1\}^{\perp}$. Furthermore, since $\L \1=0$, it follows that $\L$ is positive semidefinite. The proof is complete.
\end{proof}

\begin{thm}
All cofactors of $\L$ are equal to $2^{n-3}$.
\end{thm}
\begin{proof}
Since $\L$ is symmetric and $\L \1=0$, all cofactors of $\L$ are equal. Let the common cofactor of $\L$ be $\delta$.
By the inverse formula,
\[D^{-1}=-\frac{1}{2} \L + \frac{4}{n-1}ww'. \]
By using (P8) 
\begin{equation*}
\begin{aligned}
\det( D^{-1})&= \det (-\frac{1}{2} \L) + \frac{4}{n-1} w' \mbox{adj} (-\frac{1}{2} \L)w \\
&=\frac{4}{n-1} (-1)^{n-1} \frac{1}{2^{n-1}} \delta.
\end{aligned}
\end{equation*}
Since $n$ is even, we have
\begin{equation*}
\begin{aligned}
\det(D^{-1})=\frac{4}{1-n} \frac{1}{2^{n-1}} \delta. 
\end{aligned}
\end{equation*}
As $\det(D^{-1})=\frac{1}{1-n}$, we get $\delta=2^{n-3}$.
\end{proof}

We now obtain an interlacing property between the eigenvalues of $\L$ and $D$. 
\begin{thm}
Let the eigenvalues of $D$ and $\L$ be arranged $$\mu_1>0>\mu_2 \geq \cdots \geq \mu_n,$$ 
and
$$\lambda_1 \geq \cdots \geq \lambda_{n-1} > \lambda_n = 0,$$ respectively. Then
\[ 0>-\frac{2}{\lambda_1}\geq\mu_2\geq -\frac{2}{\lambda_2}\geq 
\cdots \geq -\frac{2}{\lambda_{n-1}}\geq \mu_n.\]
\end{thm}

\begin{proof}
The eigenvalues of $\L^{\dag}$ are 
\[0<\dfrac{1}{\lambda_1} \leq \cdots \leq \dfrac{1}{\lambda_{n-1}} .\] 
Let $Q$ be an orthogonal matrix such that
$$Q'\L^{\dag}Q=\diag\bigg(0,\dfrac{1}{\lambda_1},\dotsc,\dfrac{1}{\lambda_{n-1}}\bigg).$$
By an easy computation, we see that 
$Q' \diag(\L^{\dag})J Q$ has first column non-zero and remaining columns equal to zero. 
Since 
\[D=\diag(\L^{\dag})J+ J\diag( \L^{\dag}) -2 \L^{\dag}, \]
it follows that 
$\diag\big(-\frac{2}{\lambda_1},\dotsc,-\frac{2}{\lambda_{n-1}}\big)$ is a principal submatrix of $Q'DQ$.
By interlacing theorem, we deduce
\[\mu_1 \geq 0>-\dfrac{2}{\lambda_1} \geq \mu_2 \geq \cdots \geq -\dfrac{2}{\lambda_{n-1}} \geq \mu_n.\]
The proof is complete.
\end{proof}

\bibliography{mybibfile}
\end{document}